\documentclass{amsart}

\usepackage{amsmath, amsfonts, mathrsfs}
\usepackage{amssymb,latexsym}
\usepackage{enumerate}
\usepackage{bbm} 
\usepackage{cite}

\usepackage[top=3.1cm, bottom=3.1cm, left=3.1cm, right=3.1cm, includefoot, heightrounded]{geometry}

\newtheorem{theorem}{Theorem}[section]
\newtheorem{corollary}[theorem]{Corollary}
\newtheorem{lemma}[theorem]{Lemma}

\theoremstyle{definition}

\numberwithin{equation}{section}

\begin{document}

\title[An integral Suzuki-type fixed point theorem]
{An integral Suzuki-type fixed point theorem \\
with application}

\author{ Sokol Bush Kaliaj }

\address{
Mathematics Department, 
Science Natural Faculty, 
University of Elbasan,
Elbasan, 
Albania.
}

\email{sokolkaliaj@yahoo.com}



\begin{abstract} 
In this paper, we present an integral Suzuki-type fixed point theorem for multivalued  
mappings defined on a complete metric space in terms of  the \'{C}iri\'{c} integral contractions. 
As an application, we will prove an existence and uniqueness theorem for a functional equation arising in dynamic programming of continuous multistage decision processes.
\end{abstract}

\maketitle

\section{Introduction and Preliminaries}

The Banach contraction principle \cite{BAN} is a very famous theorem in nonlinear analysis and
has many useful applications and generalizations. 
Over the years, it has been generalized in different directions and spaces 
by several mathematicians. 
In 2008 Suzuki \cite{SUZ1} introduced a new type of mappings which generalize the well-known
Banach contraction principle. 
The same has been extended in various ways, 
see e.g. papers \cite{KIK1}-\cite{KIK4},  \cite{POP}, \cite{KAL1}, \cite{KAL2} and \cite{SUZ2}. 
Using the idea of the Kikkawa-Suzuki fixed point theorem \cite{KIK1}, 
Dori\'{c} and Lazovi\'{c} (\cite{DOR}, Theorem 2.1) 
have proved the following theorem in terms of generalized
multivalued contractions considered in \cite{CIR1}, Definition 3.  
 
\begin{theorem}\label{DOR}
Let $(X, d)$ be a complete metric space and let $T : X \to CB(X)$ be a multivalued
mapping. 
Assume that $T$ is $(r, \phi)$-Suzuki contraction with
\begin{equation}\label{eqDOR}
\phi(r)=  
\left \{
\begin{array}{ll}
1 & \text{if } 0 \leq r <\frac{1}{2} \\
1-r & \text{if } \frac{1}{2} \leq r <1 \\
\end{array} 
\right.
\end{equation}
i.e., there exists $r \in [0, 1)$ such that the implication
\begin{equation}\label{implDOR}
            \phi(r) d(x, T x) \leq d(x, y)  \Rightarrow H(T x, T y) \leq r T_{M}(x, y),
\end{equation}
holds whenever $x, y \in X$, 
where
$$
T_{M}(x, y) = 
\max
\left \{
d(x, y), d(x, Tx), d(y, Ty), \frac{d(x, Ty) + d(y, Tx)}{2}
\right \}.
$$
Then, $T$ has a fi』ed point.
\end{theorem}

In this paper, we
first define new Suzuki-type contractions in terms of  the \'{C}iri\'{c} integral contractions, 
$(r, \phi, \psi)$-Suzuki integral contractions.  
Then, a fixed point theorem for 
multivalued mappings $T : X \to CB(X)$ of $(r, \phi, \psi)$-Suzuki integral type contractions is proved. 
The main result here generalizes Theorem \ref{DOR}. 
In addition, using our
result, we proved the existence and uniqueness of solutions for a functional equation arising
in dynamic programming of continuous multistage decision processes, see Theorem \ref{T3.1}. 
Similar theorems are presented in the papers  \cite{BHAK}, \cite{KAL}, 
\cite{LIU1}, \cite{LIU2} and \cite{JAS}.

Throughout this paper, $(X, d)$ is a metric space and $\mathbb{R}^{+}$ is the set of all non-negative real
numbers, i.e., $\mathbb{R}^{+} = [0, +\infty)$.  
A mapping $\psi : \mathbb{R}^{+} \to \mathbb{R}^{+}$  is said to be \textit{subadditive}, 
if the inequality $\psi(t'+t'') \leq \psi(t')  + \psi(t'')$  
holds whenever $t',t'' \in \mathbb{R}^{+}$. 
We now set
\begin{equation*}
\begin{split} 
\Psi = \{ \psi : \mathbb{R}^{+} \to \mathbb{R}^{+} 
:& \psi \text{ is subadditive, non-decreasing and continuous on }\mathbb{R}^{+}, \\
&\psi(t) \geq t \text{ for all } t>0,\\
&\psi^{-1}(0)  = \{0\} \}.
\end{split}
\end{equation*}
The family $\Psi$ is not empty, see \cite{KAL1}.
For any $\psi \in \Psi$  and $[a, b]  \subset  \mathbb{R}^{+}$, 
by Theorem VIII.2.4, p.211, \cite{NAT}, $\psi'(t)$ exists at almost $t \in [a, b]$. 
Further, by Theorem VIII.2.5, p.212, \cite{NAT}, $\psi'$ is summable on $[a, b]$ and
$$
\int_{a}^{b} \psi'(t)dt \leq \psi(b) - \psi(a).
$$

By $CB(X)$ the family of all nonempty closed bounded subsets of $X$ is denoted. 
Let $H(\cdot, \cdot)$ be the Hausdorff metric, i.e.,
$$
H(A, B) = 
\max  
\left \{ 
\sup_{x \in A}  d(x, B), \sup_{y \in B} d(y, A) 
\right \}, 
\text{ for all } A, B \in  CB(X),
$$
where
$$
d(x, B) = \inf \{ d(x, y) : y \in B \}.
$$
It is well-known that if $(X,d)$ is a complete metric space, then 
$(CB(X), H)$ is also a complete metric space.  
We say that a multivalued mapping 
$T : X \to CB(X)$ is an \textit{$(r, \phi, \psi)$-Suzuki integral contraction},  
if $\phi$ defined by \eqref{eqDOR} 
and there exist $r \in  [0, 1)$ and $\psi \in \Psi$ such that the implication
\begin{equation}\label{implKOL}
\phi(r) \int_{0}^{d(x,T x)}\psi'(t)dt 
\leq \psi(d(x,y))
\Rightarrow 
\psi(H(Tx, Ty)) \leq r T_{\int}(x,y)
\end{equation} 
holds whenever $x, y \in X$, where
\begin{equation*}                                          
T_{\int}(x, y) = \max 
\left \{
\int_{0}^{d(x, y)} \psi'(t)dt, \int_{0}^{d(x, T x)} \psi'(t)dt, 
\int_{0}^{d(y, T y)} \psi'(t)dt, 
\int_{0}^{\frac{d(x, T y) + d(y, T x)}{2}} \psi'(t)dt           
\right \};
\end{equation*}
$T : X \to CB(X)$ is said to be \textit{$(r, \phi, \psi)$-Suzuki contraction},  
if $\phi$ defined by \eqref{eqDOR} 
and there exist $r \in  [0, 1)$ and $\psi \in \Psi$ such that the implication
\begin{equation}\label{implKOL1}
\phi(r) \psi(d(x,T x)) 
\leq \psi(d(x,y))
\Rightarrow 
\psi(H(Tx, Ty)) \leq r T_{\psi}(x,y)
\end{equation}
holds whenever $x, y \in X$, where
\begin{equation*}                                          
T_{\psi}(x, y) = \max 
\left \{
\psi(d(x, y)), \psi(d(x, T x)), 
\psi(d(y, T y)), 
\psi \left ( \frac{d(x, T y) + d(y, T x)}{2} \right )          
\right \}. 
\end{equation*}
In the special case when  
$\psi \in \Psi$ is absolutely continuous on $\mathbb{R}^{+}$ 
(i.e. $\psi$ is absolutely continuous over every interval $[0,r], (r>0)$), 
we obtain by Corollary of \cite[Theorem IX.6.2, p.264]{NAT} 
that $T$ is $(r, \phi, \psi)$-Suzuki integral contraction type 
if and only is 
$T$ is $(r, \phi, \psi)$-Suzuki contraction type. 
Note that if $\psi(t) = t$, for all $t > 0$, then 
$(r, \phi, \psi)$-Suzuki contractions coincides with 
$(r, \phi)$-Suzuki contractions.

We say that $T$  is \textit{$(r,\psi)$-\'{C}iri\'{c} integral contraction}, 
if there exist $r \in  [0, 1)$ and $\psi \in \Psi$ such that the inequality 
\begin{equation*}
\psi(H(Tx, Ty)) \leq r T_{\int}(x,y) 
\end{equation*}
holds whenever $x, y \in X$. 
Clearly, if $T$ is $(r,\psi)$-\'{C}iri\'{c} integral contraction, 
then $T$ is $(r, \phi, \psi)$-Suzuki integral contraction.
The multivalued mapping $T$ is said to have a fi』ed point, 
if there exists $z \in X$ such that $z \in T z$.

We say that a single valued mapping $S:X \to X$ 
is an \textit{$(r, \phi, \psi)$-Suzuki integral contraction}, 
if the multivalued mapping $T: X \to CB(X)$ defined as follows
$$
Tx = \{ Sx \}, \text{ for all } x \in X,
$$ 
is an $(r, \phi, \psi)$-Suzuki integral contraction. 
We say that $S$  is \textit{$(r,\psi)$-\'{C}iri\'{c} integral contraction}, 
if $T$  is $(r,\psi)$-\'{C}iri\'{c} integral contraction.
The single valued mapping $S$ is said to have a fixed point, 
if there exists $z \in X$ such that $z = S z$.

\section{The main result}

The main result is Theorem \ref{T2.1}. 
Let us start with the following  auxiliary lemma.

\begin{lemma}\label{L2.1}
Let $(X,d)$ be a complete metric space and 
let $T:X \to CB(X)$ be a multivalued mapping. 
Assume that $T$ is an $(r, \phi, \psi)$-Suzuki integral contraction. 
Then, for any $x \in X$, we have
\begin{equation}\label{eqL21}
\psi( H(Tx, Ty)) \leq 
r \int_{0}^{d(x, y)} \psi'(t)dt,
\text{ for all }y \in Tx.
\end{equation}
\end{lemma}
\begin{proof}
Fix $x \in X$ and choose an arbitrary $y \in Tx$. 
Then, 
$$d(x,Tx) \leq d(x,y),$$ 
and since $0 < \phi(r) \leq 1$, we obtain
\begin{equation}\label{eqL21.1}
\begin{split} 
\phi(r) \int_{0}^{d(x,Tx)} \psi'(t)dt \leq \int_{0}^{d(x,Tx)} \psi'(t)dt 
\leq \int_{0}^{d(x,y)} \psi'(t)dt.  
\end{split}
\end{equation}
By Theorem VIII.2.5 in \cite{NAT}, we have also
\begin{equation*}
\begin{split} 
\int_{0}^{d(x,y)} \psi'(t)dt \leq \psi(d(x,y)).  
\end{split}
\end{equation*}
The last inequality together with \eqref{eqL21.1} yields
\begin{equation*}
\begin{split} 
\phi(r) \int_{0}^{d(x,Tx)} \psi'(t)dt \leq \psi(d(x,y)).
\end{split}
\end{equation*}
Hence, by hypothesis, we obtain
\begin{equation*} 
\begin{split} 
\psi( H(Tx, Ty)) 
\leq r T_{\int}(x, y) =
r \max \{
&\int_{0}^{d(x, y)} \psi'(t)dt,\\ 
&\int_{0}^{d(x, T x)} \psi'(t)dt, \int_{0}^{d(y, T y)} \psi'(t)dt, \\
&\int_{0}^{\frac{d(x, T y) +0}{2}} \psi'(t)dt \},
\end{split}
\end{equation*}
and since
$ 
\int_{0}^{d(x,Tx)} \psi'(t)dt \leq \int_{0}^{d(x,y)} \psi'(t)dt  
$ 
and
\begin{equation*}
\begin{split} 
\frac{d(x,Ty)}{2} \leq \frac{d(x,y)}{2} + \frac{d(y, Ty)}{2} 
\leq \max \{ d(x,y), d(y, Ty) \}
\end{split}
\end{equation*}
it follows that
\begin{equation}\label{eqL21.2}
\begin{split} 
\psi( H(Tx, Ty)) 
\leq&
r \max 
\left \{
\int_{0}^{d(x, y)} \psi'(t)dt, \int_{0}^{d(y, T y)} \psi'(t)dt, 
\int_{0}^{\max \{ d(x,y), d(y, Ty) \}} \psi'(t)dt \right \}  \\
=&
r \max 
\left \{
\int_{0}^{d(x, y)} \psi'(t)dt, \int_{0}^{d(y, T y)} \psi'(t)dt          
\right \}.
\end{split}
\end{equation}
Hence, by inequality
\begin{equation*}
\begin{split} 
\psi( d(y, Ty)) \leq \psi( H(Tx, Ty))
\end{split}
\end{equation*}
we obtain 
\begin{equation*}
\begin{split} 
\psi( d(y, Ty)) \leq 
r \max 
\left \{
\int_{0}^{d(x, y)} \psi'(t)dt, \int_{0}^{d(y, T y)} \psi'(t)dt
\right \}.
\end{split}
\end{equation*} 
By Theorem VIII.2.5 in \cite{NAT}, we also have 
\begin{equation*}
\begin{split} 
\int_{0}^{d(y, T y)} \psi'(t)dt \leq  
r \max 
\left \{
\int_{0}^{d(x, y)} \psi'(t)dt, \int_{0}^{d(y, T y)} \psi'(t)dt
\right \},
\end{split}
\end{equation*}
and since $r \in [0,1)$, 
it follows that 
\begin{equation*}
\begin{split} 
\max 
\left \{
\int_{0}^{d(x, y)} \psi'(t)dt, \int_{0}^{d(y, T y)} \psi'(t)dt
\right \} = \int_{0}^{d(x, y)} \psi'(t)dt.
\end{split}
\end{equation*}
The last result together with \eqref{eqL21.2} implies
\begin{equation*}
\begin{split} 
\psi( H(Tx, Ty))  \leq r \int_{0}^{d(x, y)} \psi'(t)dt.
\end{split}
\end{equation*}
Since $y$ was arbitrary, the last result means that 
\eqref{eqL21} holds and this ends the proof.
\end{proof}

\begin{lemma}\label{L2.2}
Let $(X,d)$ be a complete metric space and 
let $T:X \to CB(X)$ be an  $(r, \phi, \psi)$-Suzuki integral contraction.
Assume that a sequence $(z_{n}) \subset X$ converges to a point $z \in X$ 
and $z_{n+1} \in Tz_{n}$ for all $n \in \mathbb{N}$. 
Then, 
\begin{equation}\label{eqL22}
\psi(d(z, Tx)) 
\leq
r \max
\left \{
\int_{0}^{d(z, x)} \psi'(t) dt, \int_{0}^{d(x, Tx)} \psi'(t) dt
\right \}, 
\text{ for all }x \in X \setminus \{z\}.
\end{equation}
\end{lemma}
\begin{proof}
Fix an arbitrary $x \in X \setminus \{z\}$. 
Then $d(z,x)>0$, 
and since  $(z_{n})$ converges to $z$, 
there exists $p \in \mathbb{N}$ such that 
$d(z_{n},z) < \frac{d(z,x)}{3}$ 
whenever $n \geq p$, 
and since $z_{n+1} \in Tz_{n}$ it follows that for each $n \geq p$, we have
\begin{equation*}
\begin{split} 
d(z_{n}, Tz_{n}) 
&\leq d(z_{n}, z_{n+1}) \leq d(z_{n}, z) + d(z, z_{n+1})\\
&\leq \frac{2}{3}d(z, x) = d(z, x) - \frac{1}{3}d(z, x)\\
&\leq d(z, x) - d(z_{n}, z) \leq d(z_{n}, x).
\end{split}
\end{equation*}
Hence, 
\begin{equation*}
\begin{split} 
\phi(r)\int_{0}^{d(z_{n}, Tz_{n})} \psi'(t)dt 
&\leq \int_{0}^{d(z_{n}, Tz_{n})} \psi'(t)dt 
\leq \int_{0}^{d(z_{n}, x)} \psi'(t)dt \\
&\leq \psi(d(z_{n}, x)),
\end{split}
\end{equation*}
whenever $n \geq p$. 
By hypothesis the last result implies 
\begin{equation*}
\begin{split} 
\psi( H(Tz_{n}, Tx)) 
\leq r T_{\int}(z_{n}, x) =
r 
\max
\{&
\int_{0}^{d(z_{n}, x)} \psi'(t) dt, \\
&\int_{0}^{d(z_{n}, Tz_{n})} \psi'(t) dt, \int_{0}^{d(x, Tx)} \psi'(t) dt, \\
&\int_{0}^{\frac{d(z_{n}, Tx) + d(x, Tz_{n})}{2}} \psi'(t) dt 
\},
\end{split}
\end{equation*}
whenever $n \geq p$. 
Further, by inequalities 
\begin{equation*}
\begin{split} 
\psi( d(z_{n+1}, Tx)) \leq \psi( H(Tz_{n}, Tx)),  ~ 
\int_{0}^{d(z_{n}, Tz_{n})} \psi'(t) dt \leq \psi(d(z_{n}, Tz_{n})) \leq \psi(d(z_{n}, z_{n+1}))
\end{split}
\end{equation*}
and 
\begin{equation*}
\begin{split} 
\frac{d(z_{n}, Tx) + d(x, Tz_{n})}{2} \leq 
\max \{ d(z_{n}, Tx), d(x, Tz_{n}) \} 
\leq \max \{ d(z_{n}, Tx), d(x, z_{n+1}) \} 
\end{split}
\end{equation*}
it follows that
\begin{equation*}
\begin{split} 
\psi( d(z_{n+1}, Tx)) 
\leq r 
\max
\{&
\int_{0}^{d(z_{n}, x)} \psi'(t) dt, \\ 
&\int_{0}^{d(z_{n}, z_{n+1})} \psi'(t) dt, \int_{0}^{d(x, Tx)} \psi'(t) dt, \\
&\int_{0}^{d(z_{n}, Tx)} \psi'(t) dt, \int_{0}^{d(x, z_{n+1})} \psi'(t) dt 
\}
\end{split}
\end{equation*}
whenever $n \geq p$. 
Hence, by Theorem IX.4.1 in \cite{NAT}, p.252, we get
\begin{equation*}
\begin{split} 
\psi( d(z, Tx)) &= \lim_{n \to \infty}\psi( d(z_{n+1}, Tx)) \\
&\leq r 
\max
\left \{
\int_{0}^{d(z, x)} \psi'(t) dt, \int_{0}^{d(x, Tx)} \psi'(t) dt,
\int_{0}^{d(z, Tx)} \psi'(t) dt
\right \},
\end{split}
\end{equation*}
and since $\int_{0}^{d(z, Tx)} \psi'(t) dt \leq \psi(d(z, Tx))$ we obtain
\begin{equation*}
\begin{split} 
\int_{0}^{d(z, Tx)} \psi'(t) dt \leq r 
\max
\left \{
\int_{0}^{d(z, x)} \psi'(t) dt, \int_{0}^{d(x, Tx)} \psi'(t) dt,
\int_{0}^{d(z, Tx)} \psi'(t) dt
\right \}.
\end{split}
\end{equation*}
The last result implies
\begin{equation*}
\begin{split} 
\max
\left \{
\int_{0}^{d(z, x)} \psi'(t) dt, \int_{0}^{d(x, Tx)} \psi'(t) dt,
\int_{0}^{d(z, Tx)} \psi'(t) dt
\right \}
=
\max
\{
&\int_{0}^{d(z, x)} \psi'(t) dt, \\
&\int_{0}^{d(x, Tx)} \psi'(t) dt
\}
\end{split}
\end{equation*}
and consequently
\begin{equation*}
\begin{split} 
\psi( d(z, Tx)) 
\leq& r 
\max
\left \{
\int_{0}^{d(z, x)} \psi'(t) dt, \int_{0}^{d(x, Tx)} \psi'(t) dt
\right \}.
\end{split}
\end{equation*}
Since $x$ was arbitrary, 
the last result yields that \eqref{eqL22} holds and the proof is finished.
\end{proof}

We are now ready to present the main result.

\begin{theorem}\label{T2.1}
Let $(X,d)$ be a complete metric space and 
let $T:X \to CB(X)$ be a multivalued mapping. 
Assume that $T$ is an $(r, \phi, \psi)$-Suzuki integral contraction. 
Then, $T$ has a fixed point.
\end{theorem}
\begin{proof}
Let us start with an arbitrary point $z_{0} \in X$. 
Choose a real number $\overline{r}$ such that $r < \overline{r} < 1$.

If $d(z_{0}, Tz_{0}) = 0$, then $z_{0}$ is a fixed point and the proof is finished.
Assume that $d(z_{0}, Tz_{0})>0$. 
Then, there exists  $z_{1} \in Tz_{0}$ with 
$d(z_{0}, z_{1}) \geq d(z_{0}, Tz_{0}) >0$. 
Since $z_{1} \in Tz_{0}$, by Lemma \ref{L2.1}, we have
\begin{equation*}
\begin{split} 
\psi( d(z_{1}, T z_{1}) ) \leq \psi( H( Tz_{0}, Tz_{1}) ) 
\leq& r \int_{0}^{d(z_{0}, z_{1})} \psi'(t) dt \\
\leq& r \psi( d(z_{0}, z_{1}) ) < \overline{r} \psi( d(z_{0}, z_{1}) ).
\end{split}
\end{equation*}
We assume that $r >0$ since
$$
r=0 \Rightarrow \psi( d(z_{1}, T z_{1}) ) = 0 
\Rightarrow d(z_{1}, T z_{1})  = 0 
\Rightarrow  z_{1} \in T z_{1}.
$$

If $d(z_{1}, Tz_{1})=0$, then $z_{1}$ is a fixed point and the proof is finished. 
Assume that $d(z_{1}, Tz_{1})=t_{1}>0$. 
Since $\psi$ is continuous at $t_{1}$, 
given 
$$
\varepsilon_{1} = \overline{r} \psi( d(z_{0}, z_{1}) ) -\psi( t_{1} ) >0
$$ 
there exists $\delta_{1}>0$ such that 
\begin{equation*}
t_{1} \leq t < t_{1} + \delta_{1} 
\Rightarrow 
\psi(t_{1}) \leq \psi(t) < \overline{r} \psi( d(z_{0}, z_{1}) ),
\end{equation*}
and since there exists $z_{2} \in Tz_{1}$ such that 
\begin{equation*}
t_{1} \leq d(z_{1}, z_{2}) < t_{1} + \delta_{1} 
\end{equation*}
it follows that 
\begin{equation}\label{eqT21.1} 
\psi(t_{1}) \leq \psi(d(z_{1}, z_{2})) < \overline{r} \psi( d(z_{0}, z_{1}) ).
\end{equation}
We now assume that $z_{n} \in Tz_{n-1}$ has been chosen. 
Then, by Lemma \ref{L2.1}, we have
\begin{equation*}
\begin{split} 
\overline{r} \psi( d(z_{n-1}, z_{n}) )-\psi( d(z_{n}, T z_{n}) )>0.
\end{split}
\end{equation*}
If $d(z_{n}, Tz_{n})=0$, then $z_{n}$ is a fixed point and the proof is finished. 
Assume that $d(z_{n}, Tz_{n})=t_{n}>0$. 
Since $\psi$ is continuous at $t_{n}$, 
given 
$$
\varepsilon_{n} = \overline{r} \psi( d(z_{n-1}, z_{n}) ) -\psi( t_{n} ) >0
$$ 
there exists $\delta_{n}>0$ such that 
\begin{equation*}
t_{n} \leq t < t_{n} + \delta_{n} 
\Rightarrow 
\psi(t_{n}) \leq \psi(t) < \overline{r} \psi( d(z_{n-1}, z_{n}) ),
\end{equation*}
and since there exists $z_{n+1} \in Tz_{n}$ such that 
\begin{equation*}
t_{n} \leq d(z_{n}, z_{n+1}) < t_{n} + \delta_{n} 
\end{equation*}
it follows that 
\begin{equation}\label{eqT21.n}
\psi(t_{n}) \leq \psi(d(z_{n}, z_{n+1})) < \overline{r} \psi( d(z_{n-1}, z_{n}) ).
\end{equation}
Since $z_{n+1} \in Tz_{n}$, by Lemma \ref{L2.1}, we have
\begin{equation*}
\begin{split} 
\overline{r} \psi( d(z_{n}, z_{n+1}) )-\psi( d(z_{n+1}, T z_{n+1}) )>0.
\end{split}
\end{equation*}

By above construction we obtain a sequence $(z_{n}) \subset X$ such that 
\begin{equation}\label{eqT21.C}
\psi(d(z_{n}, z_{n+1})) < \overline{r} \psi( d(z_{n-1}, z_{n}) ),
\text{ for all }n \in \mathbb{N}.
\end{equation}
Hence, we get
\begin{equation*}
\begin{split} 
\psi(d(z_{n}, z_{n+1})) < \overline{r}^{n} \psi( d(z_{0}, z_{1}) ),
\text{ for all }n \in \mathbb{N},
\end{split}
\end{equation*}
and since $\psi(d(z_{n}, z_{n+1})) \geq d(z_{n}, z_{n+1})$ 
it follows that 
\begin{equation*}
\begin{split} 
d(z_{n}, z_{n+1}) < \overline{r}^{n} \psi( d(z_{0}, z_{1}) ),
\text{ for all }n \in \mathbb{N}.
\end{split}
\end{equation*}
Hence 
\begin{equation*}
\begin{split} 
\sum_{n=1}^{+\infty} d(z_{n}, z_{n+1}) 
\leq 
\sum_{n=1}^{+\infty} \overline{r}^{n} \psi( d(z_{0}, z_{1}) ).
\end{split}
\end{equation*} 
Since $\overline{r} < 1$, the last result yields that $(z_{n})$ is a Cauchy sequence in $X$ 
and by completeness of $X$ it follows that  $(z_{n})$ converges to a point 
$z \in X$.

We are going to prove that $z$ is a fixed point of $T$. 
To see this, we suppose that $z \not\in Tz$. 
Consider the following possible cases:
\begin{itemize}
\item[(i)]
$0 \leq r < \frac{1}{2}$,
\item[(ii)]
$\frac{1}{2} \leq r <1$.
\end{itemize}

$(i)$  
Fix an arbitrary $w \in Tz$. 
Since
$$
d(z, Tz) \leq d(z,Tw) + H(Tw, Tz)
$$
we obtain
\begin{equation}\label{eqRasti1.1}
\begin{split} 
\int_{0}^{d(z, Tz)} \psi'(t)dt 
\leq&
\psi(d(z, Tz)) 
\leq 
\psi(d(z,Tw)) + \psi(H(Tw, Tz)),
\end{split}
\end{equation}
and since by Lemma \ref{L2.1} and Lemma \ref{L2.2} we have also
\begin{equation*}
\begin{split} 
\psi(H(Tw, Tz)) 
\leq& 
r \int_{0}^{d(w,z)} \psi'(t)dt \\
\psi(d(z,Tw)) 
\leq& 
r \max 
\left \{
\int_{0}^{d(z,w)} \psi'(t)dt, \int_{0}^{d(w,Tw)} \psi'(t)dt,
\right \} 
\end{split}
\end{equation*} 
and 
\begin{equation*}
\begin{split} 
\int_{0}^{d(w,Tw)}\psi'(t)dt \leq \psi(d(w,Tw)) \leq \psi(H(Tz,Tw)) 
&\leq r \int_{0}^{d(z,w)} \psi'(t)dt \\
&< \int_{0}^{d(z,w)} \psi'(t)dt
\end{split}
\end{equation*} 
it follows that
\begin{equation*}
\begin{split} 
\int_{0}^{d(z, Tz)} \psi'(t)dt 
\leq
\psi(d(z, Tz))  \leq 2r \int_{0}^{d(z,w)} \psi'(t)dt. 
\end{split}
\end{equation*}
Since $w$ was arbitrary, 
the last results yields
\begin{equation}\label{eqRasti1.2}
\begin{split} 
\int_{0}^{d(z, Tz)} \psi'(t)dt \leq
\psi(d(z, Tz)) 
\leq 2r \int_{0}^{d(z,w)} \psi'(t)dt, \text{ for all }w \in Tz. 
\end{split}
\end{equation}
From equality 
$$
d(z,Tz) = \inf \{d(z,w): w \in Tz \}
$$
it follows that 
there exists a sequence $(w_{n}) \subset Tz$ such that
$$
\lim_{n \to \infty} d(z,w_{n}) = d(z, Tz).
$$
Then, by \eqref{eqRasti1.2}, we obtain
\begin{equation*}
\begin{split} 
\int_{0}^{d(z, Tz)} \psi'(t)dt \leq
\psi(d(z, Tz)) 
\leq 2r \int_{0}^{d(z,w_{n})} \psi'(t)dt,  
\text{ for all }n \in \mathbb{N}.
\end{split}
\end{equation*}
Hence, by Theorem IX.4.1 in \cite{NAT}, we get
\begin{equation*}
\begin{split} 
\int_{0}^{d(z, Tz)} \psi'(t)dt 
\leq
\psi(d(z, Tz)) 
\leq 2r \int_{0}^{d(z,Tz)} \psi'(t)dt 
\end{split}
\end{equation*}
and since $0< 2r <1$ it follows that
\begin{equation*}
\begin{split} 
\int_{0}^{d(z, Tz)} \psi'(t)dt = 0 \Rightarrow \psi(d(z, Tz)) = 0 
\Rightarrow d(z,Tz) = 0 
\end{split}
\end{equation*}
or
\begin{equation*}
\begin{split} 
\int_{0}^{d(z, Tz)} \psi'(t)dt > 0 \Rightarrow 
\int_{0}^{d(z, Tz)} \psi'(t)dt < \int_{0}^{d(z, Tz)} \psi'(t)dt.
\end{split}
\end{equation*}
These contradictions show that $z \in Tz$.

$(ii)$
Fix an arbitrary $x \in X \setminus \{z\}$. 
Since
$$
d(x, Tx) \leq d(x,z) + d(z, Tx)
$$
we get
\begin{equation}\label{eqRasti2.1}
\begin{split} 
\int_{0}^{d(x, Tx)} \psi'(t)dt 
\leq&
\int_{0}^{d(x,z) + d(z,Tx)} \psi'(t)dt \\
\leq&
\int_{0}^{d(z,Tx)} \psi'(t)dt + \int_{d(z,Tx)}^{d(x,z) + d(z,Tx)} \psi'(t)dt \\
\leq&
\int_{0}^{d(z,Tx)} \psi'(t)dt + \psi[d(z,x) + d(z,Tx)] - \psi(d(z,Tx)) \\
\leq&
\int_{0}^{d(z,Tx)} \psi'(t)dt + \psi(d(z,x)).
\end{split}
\end{equation}
Since $x \neq z$, by Lemma \ref{L2.2}, we have
\begin{equation}\label{eqRasti2.2}
\begin{split} 
\int_{0}^{d(z,Tx)} \psi'(t)dt 
&\leq \psi( d(z, Tx) ) \\
&\leq r \max 
\left \{
\int_{0}^{d(z,x)} \psi'(t)dt, \int_{0}^{d(x,Tx)} \psi'(t)dt
\right \}.
\end{split}
\end{equation}  
If $\int_{0}^{d(x,Tx)} \psi'(t)dt \leq \int_{0}^{d(z,x)} \psi'(t)dt$, then 
\begin{equation}\label{eqRasti2.3}
\begin{split} 
\phi(r) \int_{0}^{d(x, Tx)} \psi'(t)dt \leq \int_{0}^{d(z,x)} \psi'(t)dt 
\leq \psi(d(z,x))
\end{split}
\end{equation}
Otherwise, if $\int_{0}^{d(x,Tx)} \psi'(t)dt > \int_{0}^{d(z,x)} \psi'(t)dt$, 
then we obtain by \eqref{eqRasti2.2} that
\begin{equation*}
\begin{split} 
\int_{0}^{d(z,Tx)} \psi'(t)dt 
\leq \psi( d(z, Tx) ) 
\leq r \int_{0}^{d(x,Tx)} \psi'(t)dt,
\end{split}
\end{equation*}
and by \eqref{eqRasti2.1} we get
\begin{equation*}
\begin{split} 
\phi(r) \int_{0}^{d(x, Tx)} \psi'(t)dt = (1-r)\int_{0}^{d(x,Tx)} \psi'(t)dt \leq \psi( d(z, x) ).
\end{split}
\end{equation*}
Since $x$ was arbitrary, the last result together with \eqref{eqRasti2.3} yields
\begin{equation*}
\begin{split} 
\phi(r) \int_{0}^{d(x, Tx)} \psi'(t)dt \leq \psi( d(z, x) ), 
\text{ for all }x \in X \setminus \{z\}.
\end{split}
\end{equation*}
Then, by hypothesis it follow that
\begin{equation*}
\begin{split} 
\psi( H(Tx,Tz) ) \leq r \max \{  
&\int_{0}^{d(z,x)} \psi'(t)dt,\\
&\int_{0}^{d(z,Tz)} \psi'(t)dt, \int_{0}^{d(x,Tx)} \psi'(t)dt,\\
&\int_{0}^{\max\{d(z,Tx), d(x,Tz)\}} \psi'(t)dt \},
\end{split}
\end{equation*}
whenever $x \in X \setminus \{z\}$. 
Clearly, if $x=z$, then the last inequality also holds. 
Thus, the last inequality holds for all $x \in X$. 
In particular, for $x =z_{n}$, we have
\begin{equation*}
\begin{split} 
\psi( d(z_{n+1},Tz) ) \leq \psi( H(Tz_{n},Tz) ) \leq r \max \{  
&\int_{0}^{d(z,z_{n})} \psi'(t)dt,\\
&\int_{0}^{d(z,Tz)} \psi'(t)dt, \int_{0}^{d(z_{n}, z_{n+1})} \psi'(t)dt,\\
&\int_{0}^{\max\{d(z, z_{n+1}), d(z_{n},Tz)\}} \psi'(t)dt \},
\end{split}
\end{equation*} 
Hence, by Theorem IX.4.1 in \cite{NAT}, it follows that
\begin{equation*}
\begin{split}  
0< \psi( d(z, Tz) ) 
=& \lim_{n \to \infty} \psi( d(z_{n+1},Tz) ) \\
\leq& 
r \int_{0}^{d(z,Tz)} \psi'(t)dt \leq r \psi( d(z, Tz) )<\psi( d(z, Tz) ).
\end{split}
\end{equation*} 
This contradiction shows that $z \in Tz$ and the proof is finished.
\end{proof}

\begin{corollary}\label{C2.1} 
Let $(X,d)$ be a complete metric space and 
let $T:X \to CB(X)$ be a multivalued mapping. 
If $T$  is an $(r,\psi)$-\'{C}iri\'{c} integral contraction, then $T$ has a  fixed point.
\end{corollary}

\begin{corollary}\label{C2.2}
Let $(X,d)$ be a complete metric space and 
let $S:X \to X$ be a mapping. 
If $S$ is an $(r, \phi, \psi)$-Suzuki integral contraction, 
then $S$ has a unique fixed point.
\end{corollary}
\begin{proof}
Since the multivalued mapping $T : X \to CB(X)$ defined as follows
$$
Tx = \{ Sx \}, \text{ for all } x \in X,
$$
is an $(r, \phi, \psi)$-Suzuki integral contraction,  
by Theorem \ref{T2.1}, $T$ has a fixed point $z$
and, consequently, $z$ is a fixed point for $S$.

It remains to prove that the fixed point $z$ is unique. 
Assume that $z' \in X$ is another fixed point for $S$ 
and $z \neq z'$. 
Then, by Lemma \ref{L2.2}, we have
$$
\int_{0}^{d(z,Tz')} \psi'(t) dt \leq \psi(d(z,Tz')) \leq  
r \max
\left \{
\int_{0}^{d(z,z')} \psi'(t) dt, \int_{0}^{d(z',Tz')} \psi'(t) dt,
\right \}
$$  
and since  $\int_{0}^{d(z',Tz')} \psi'(t) dt = 0$ 
and $d(z,Tz') = d(z,z')$ 
it follows that
$$
0 \leq \int_{0}^{d(z, z')} \psi'(t) dt \leq \psi(d(z, z')) \leq  
r \int_{0}^{d(z,z')} \psi'(t) dt 
$$
Hence,
\begin{equation*}
\begin{split} 
\left ( r = 0 \text{ or }\int_{0}^{d(z,z')} \psi'(t) dt = 0 \right ) 
\Rightarrow \psi(d(z, z')) = 0 \Rightarrow d(z, z') = 0,
\end{split}
\end{equation*} 
or
\begin{equation*}
\begin{split} 
\left ( 0 < r <1 \text{ and }\int_{0}^{d(z,z')} \psi'(t) dt > 0 \right )  \Rightarrow \int_{0}^{d(z, z')} \psi'(t) dt < 
\int_{0}^{d(z,z')} \psi'(t) dt. 
\end{split}
\end{equation*}
These contradicts show that $z =z'$ and this ends the proof.
\end{proof}

\begin{corollary}\label{C2.22}
Let $(X,d)$ be a complete metric space and 
let $S:X \to X$ be a mapping. 
If $S$ is an $(r, \psi)$-\'{C}iri\'{c} integral contraction, then $S$ has a unique fixed point.
\end{corollary}

\begin{corollary}\label{C2.3}
Let $(X,d)$ be a complete metric space and 
let $T:X \to CB(X)$ be a multivalued mapping. 
If $T$  is an $(r,\phi)$-Suzuki multivalued contraction, then $T$ has a  fixed point.
\end{corollary}
\begin{proof}
Clearly, $T$ is an $(r, \phi, \psi)$-Suzuki integral contraction 
with $\psi(t)=t$ for all $t \in \mathbb{R}^{+}$. 
Then, by Theorem \ref{T2.1}, $T$ has a fixed point and the proof is finished.
\end{proof}

\section{An Application to Dynamic Programming}

We will now prove an existence and uniqueness theorem for a functional equation arising in dynamic programing of continuous multistage decision processes. 
From now on $X$ and $Y$ will be the Banach spaces. 
Let $S \subset X$ be the \textit{state space} and 
let $D \subset Y$ be the \textit{decision space}. 
In the paper \cite{BELL}, Bellman and Lee gave the following basic form 
of the functional equation of dynamic programming:
$$
f(x) = \sup_{y \in D} H(x,y,f[T(x,y)]),
$$
where $x$ and $y$ represent the state and decision vectors respectively, 
$T: S \times D \to S$ represents the transformation of the process 
and $f(x)$ represent the optimal return function with initial state $x$. 

We will study the existence and the uniqueness of the solution 
of the following functional equation:
\begin{equation}\label{eqKryesor}
f(x) = \sup_{y \in D} [~ g(x, y) + G(x, y, f[T(x,y)])~], \text{ for all }x \in S, 
\end{equation}   
where $g: S\times D \to \mathbb{R}$ and 
$g: S\times D \times \mathbb{R} \to \mathbb{R}$ are bounded functions.

Let $B(S)$ be the Banach space of real valued bounded functions defined on $S$ with norm
$$
||h|| = \sup\{|h(x)|: x \in S\}, \text{ for all } h \in B(S). 
$$ 
Since $g$ and $G$ are bounded functions we can define a mapping 
$A: B(S) \to B(S)$ as follows
\begin{equation}\label{mappKryesor}
(Ah)(x) = \sup_{y \in D} [~ g(x, y) + G(x, y, h[T(x,y)])~], 
\end{equation}
for each $h \in B(S)$ and $x \in S$. 
It is easy to see that any fixed point of $A$ is a solution 
of the functional equation \eqref{eqKryesor} 
and, conversely, 
any bounded solution of  \eqref{eqKryesor} 
is a fixed point of $A$.

Clearly, if $A$ is an $(r, \phi, \psi)$- Suzuki integral contraction, 
then there exists $r \in [0,1)$ and $\psi \in \Psi$ such that the implication
$$
\phi(r) \int_{0}^{||h-Ah||} \psi'(t)dt \leq \psi(||h-\ell||) 
\Rightarrow 
\psi(||Ah-A\ell||) \leq r A_{\int}(h,\ell),
$$
holds whenever $h, \ell \in B(S)$, 
where $\phi$ is defined by \eqref{eqDOR} 
and
\begin{equation*}
\begin{split}                                          
A_{\int}(h,\ell) = \max 
\{
&\int_{0}^{||h-\ell||} \psi'(t)dt, \int_{0}^{||h-Ah||} \psi'(t)dt, 
\int_{0}^{||\ell-A\ell||} \psi'(t)dt, \\
&\int_{0}^{\frac{||h-A\ell||) + ||\ell-Ah||}{2}} \psi'(t)dt           
\}.
\end{split}
\end{equation*}

Since $\Psi$ is a subfamily of $\mathfrak{F}$ defined in \cite{KAL}, 
we obtain by Lemmas 3.1 and 3.2 in \cite{KAL} 
the following auxiliary lemmas.

\begin{lemma}\label{L3.1}
Let $R$ be a subset of $\mathbb{R}^{+}$ such that $\sup R < +\infty$, 
and let $\psi \in \Psi$. 
Then,
$$
\psi(\sup R) = \sup \psi(R).
$$
\end{lemma}

\begin{lemma}\label{L3.2}
Assume that the functions $g$ and $G$ are bounded and 
let $\varepsilon >0$, $x \in S$, $h, \ell \in B(S)$ 
and let $\psi \in \Psi$. 
Then, there exists $y_{1}, y_{2} \in D$ such that
\begin{equation}\label{eqL32}
\psi( |(Ah)(x)-(A\ell)(x)| ) 
\leq
\max \{ \psi(|a(x,y_{1})|), \psi(|b(x,y_{2})|) \} + \varepsilon,  
\end{equation}
where
\begin{equation*}
\begin{split} 
a(x,y_{1}) &= G(x, y_{1}, h[T(x,y_{1})]) - G(x, y_{1}, \ell[T(x,y_{1})]),\\
b(x,y_{2}) &= G(x, y_{2}, h[T(x,y_{2})]) - G(x, y_{2}, \ell[T(x,y_{2})]).
\end{split}
\end{equation*} 
\end{lemma}  

We are now ready to prove the existence and uniqueness of the solution of the functional equation 
\eqref{eqKryesor} under certain conditions.

\begin{theorem}\label{T3.1}
Assume that the following conditions are satisfied:
\begin{itemize}
\item[(i)]
$g$ and $G$ are bounded functions,
\item[(ii)]
there exist $r \in [0,1)$ and $\psi \in \Psi$ such that 
for each $h, \ell \in B(S)$, the inequality
$$
\phi(r) \int_{0}^{||h - Ah||} \psi'(t)dt \leq \psi(||h - \ell||)
$$
implies that for each $(x,y) \in S \times D$, we have
$$
\psi(~ |G(x,y,h[T(x,y)]) - G(x,y,\ell[T(x,y)])| ~) \leq 
r A_{\int}(h,\ell). 
$$
\end{itemize}
Then, the functional equation \eqref{eqKryesor} 
possesses 
unique bounded solution on $S$. 
\end{theorem}
\begin{proof}
Assume that an arbitrary $\varepsilon>0$ and $h, \ell \in B(S)$ such that
\begin{equation}\label{eqT31.1}
\phi(r) \int_{0}^{||h-Ah||} \psi'(t) dt \leq \psi(||h-\ell||)
\end{equation}
are given. 
Fix an arbitrary $x \in S$. 
Then, by Lemma \ref{L3.2}, there exists $y_{1}, y_{2} \in D$ such that
\begin{equation*}
\psi( |(Ah)(x)-(A\ell)(x)| ) 
\leq
\max \{ \psi(|a(x,y_{1})|), \psi(|b(x,y_{2})|) \} + \varepsilon.  
\end{equation*}
By hypothesis, we also have
\begin{equation*}
\begin{split} 
\psi(|a(x,y_{1})|) \leq r A_{\int}(h,\ell)
\quad\text{and}\quad
\psi(|b(x,y_{2})|) \leq r A_{\int}(h,\ell).
\end{split}
\end{equation*}
It follows that 
\begin{equation*}
\psi( |(Ah)(x)-(A\ell)(x)| ) 
\leq
r A_{\int}(h,\ell) + \varepsilon.  
\end{equation*}
Since $x$ was arbitrary, the last inequality holds for all $x \in S$. 
Hence, 
\begin{equation*}
\sup_{x \in S} \psi( |(Ah)(x)-(A\ell)(x)| ) 
\leq
r A_{\int}(h,\ell) + \varepsilon,  
\end{equation*}
and since by Lemma \ref{L3.1} we also have
\begin{equation*}
\sup_{x \in S} \psi( |(Ah)(x)-(A\ell)(x)| ) =
\psi( \sup_{x \in S} |(Ah)(x)-(A\ell)(x)| ) = 
\psi( ||Ah-A\ell|| ),
\end{equation*}
it follows that
\begin{equation*}
\psi( ||Ah-A\ell|| )
\leq
r A_{\int}(h,\ell) + \varepsilon. 
\end{equation*}
Since $\varepsilon >0$ was arbitrary, 
the last result implies
\begin{equation}\label{eqT31.2}
\psi( ||Ah-A\ell|| )
\leq
r A_{\int}(h,\ell). 
\end{equation}
Thus, inequality \eqref{eqT31.1} implies  \eqref{eqT31.2}, 
whenever $h, \ell \in B(S)$. 
This means that $A$ is an $(r, \phi, \psi)$-Suzuki integral contraction. 
Therefore, by Corollary \ref{C2.2}, 
$A$ has a unique fixed point and, therefore, 
the functional equation \eqref{eqKryesor} 
possesses unique bounded solution on $S$, 
and this ends the proof.
\end{proof}

\bibliographystyle{plain}

\end{document}